\title{A note on embracing exchange sequences in oriented matroids}

\documentclass[11pt]{article}
\usepackage{fullpage}

\usepackage[utf8]{inputenc}
\usepackage{amsmath}
\usepackage{amsthm,amssymb}
\usepackage{nicefrac,amsfonts}
\usepackage{thmtools,mathtools,leftindex,leftidx}
\usepackage{caption,subcaption}
\usepackage{epsfig,graphicx,graphics,color,tikz,tikz-cd,tcolorbox}
\usepackage{enumerate,enumitem}
\usepackage[sort,nocompress]{cite}
\usepackage{xspace}
\usepackage{bbm}
\usepackage{hyperref}
\usepackage[capitalise]{cleveref}
\usepackage{float}
\hypersetup{
    colorlinks=true,       
    linkcolor=blue,        
    citecolor=magenta,         
    filecolor=magenta,     
    urlcolor=cyan,         
    linktoc=all
}
\usepackage{titling}

\newlength{\bibitemsep}
\newlength{\bibparskip}
\let\oldthebibliography\thebibliography
\renewcommand\thebibliography[1]{%
  \oldthebibliography{#1}%
  \setlength{\parskip}{\bibitemsep}%
  \setlength{\itemsep}{\bibparskip}%
}

\makeatletter
\renewcommand{\paragraph}{%
  \@startsection{paragraph}{4}%
  {\z@}{1.6ex \@plus 1ex \@minus .2ex}{-0.5em}%
  {\normalfont\normalsize\bfseries}%
}
\makeatother

\theoremstyle{plain}
\newtheorem{thm}{Theorem}[section]

\newtheorem{cla}[thm]{Claim}

\newtheorem{conj}[thm]{Conjecture}

\theoremstyle{definition}

\newtheorem{ex}[thm]{Example}

\newcommand*{\claimproofname}{Proof of claim.}

\makeatletter
\newcommand{\leqnomode}{\tagsleft@true}
\newcommand{\reqnomode}{\tagsleft@false}
\makeatother

\makeatletter
 \newcommand{\linkdest}[1]{\Hy@raisedlink{\hypertarget{#1}{}}}
\makeatother

\def\final{0}  
\def\iflong{\iffalse}
\ifnum\final=0  
\newcommand{\kristof}[1]{{\color{red}[{ \textbf{Kristóf:}  #1}]\marginpar{\color{red}*}}}
\newcommand{\benedek}[1]{{\color{blue}[{ \textbf{Benedek:}  #1}]\marginpar{\color{blue}*}}}
\else 
\newcommand{\kristof}[1]{}
\newcommand{\benedek}[1]{}
\fi

\newcommand{\bR}{\mathbb{R}}

\newcommand{\cB}{\mathcal{B}}
\newcommand{\cC}{\mathcal{C}}

\let\Right\bigr
\let\Left\bigl
\makeatletter
\def\bigr#1{\Right#1\@ifnextchar){\!\bigr}{}}
\def\bigl#1{\Left#1\@ifnextchar({\!\bigl}{}}
\makeatother

\thanksmarkseries{alph}
\author{
Kristóf Bérczi\thanks{MTA-ELTE Matroid Optimization Research Group and HUN-REN-ELTE Egerváry Research Group, Department of Operations Research, ELTE Eötvös Loránd University, and HUN-REN Alfréd Rényi Institute of Mathematics, Budapest, Hungary. Email: \texttt{kristof.berczi@ttk.elte.hu}.}
\and
Benedek Nádor\thanks{Department of Operations Research, ELTE Eötvös Loránd University, Budapest, Hungary. Email: \texttt{nador.benedek@gmail.com}.} 
}
\date{}

\begin{document}
\maketitle

\thispagestyle{empty}
\begin{abstract} 
An open problem in convex geometry asks whether two simplices $A,B\subseteq\bR^d$, both containing the origin in their convex hulls, admit a polynomial-length sequence of vertex exchanges transforming $A$ into $B$ while maintaining the origin in the convex hull throughout. We propose a matroidal generalization of the problem to oriented matroids, concerning exchange sequences between bases under sign constraints on elements appearing in certain fundamental circuits. We formulate a conjecture on the minimum length of such a sequence, and prove it for oriented graphic matroids of directed graphs. We also study connections between our conjecture and several long-standing open problems on exchange sequences between pairs of bases in unoriented matroids.
\medskip

\end{abstract}
\section{Introduction}
\label{sec:intro}

Carathéodory's theorem~\cite{caratheodory1911variabilitatsbereich} is a cornerstone of convex geometry, describing how points in a convex hull can be represented. It states that if a point $x$ lies in the convex hull of a set $P\subseteq\bR^d$, then there exists a subset of at most $d+1$ points of $P$ whose convex hull already contains $x$. In other words, every point in a convex hull admits a sparse representation, with the bound $d+1$ being best possible. A fundamental extension of this classical result is the Colorful Carathéodory theorem, introduced by Bárány~\cite{barany1982generalization}. In this setting, one is given $d+1$ point sets, called color classes, each containing the origin in its convex hull. The theorem guarantees that one can select a single point from each color class so that the origin lies in the convex hull of the $d+1$ chosen points. Equivalently, the colorful version ensures the existence of a ``rainbow'' simplex embracing the origin. Both the classical and colorful Carathéodory theorems have played a central role in combinatorial geometry, optimization, and matroid theory, and continue to motivate algorithmic and structural questions.

From an algorithmic point of view, there is a significant difference between the classical and the colorful variants. The classical theorem can be realized using standard linear programming: a basic feasible solution of the corresponding linear system yields a convex combination involving at most $d+1$ points in polynomial time. In contrast, constructive proofs of the colorful variant rely on iterative linear programming steps, and no strongly polynomial-time algorithm is known for the general case, although special cases and approximate variants admit more efficient procedures; we refer the interested reader to~\cite{mulzer2018computational} for further details.

A natural relaxation of this problem arises when there are only $k<d+1$ point sets, and the goal is to find a colourful set containing, for each $i \in [k]$, $d_i$ points from the $i$-th set. When $k=2$ and the target sizes $(d_1,d_2)$ satisfy $d_1+d_2 = d+1$, the problem admits a polynomial-time solution, as shown by Meunier, Mulzer, Sarrabezolles, and Stein~\cite{meunier2017rainbow}. In an effort to better understand this 2-colour case and to explore the possibility of extending polynomial-time algorithms to higher numbers of colours, Caoduro, Khodamoradi, Paat, and Shepherd~\cite{marco2024} raised the following problem: {\it Given two 
0-embracing sets $A,B\subseteq\bR^d$, that is, sets whose convex hulls contain the origin, is it always possible to transform $A$ into $B$ by a polynomial-length sequence of exchanges, where each step replaces a point of the current set with a point from $A\cup B$ and every intermediate set remains 0-embracing?} We call the minimum number of such exchanges the {\it $0$-embracing exchange distance} of $A$ and $B$. 

To see this, form a $(d+1)\times(2d+2)$ matrix $M$ whose columns are the coordinates of the points in $A\cup B$, augmented by an additional row consisting entirely of $1$'s. Let $b\in\mathbb{R}^{d+1}$ be the vector whose first $d$ entries are $0$ and whose last entry is $1$. Each exchange of a point in $A$ with a point in $B$ can then be interpreted as traversing an edge of the polyhedron $P=\{x\in\mathbb{R}^{2d+2}\mid Mx=b,; x\ge 0\}$. Therefore, if $G$ denotes the $1$-skeleton of $P$, then following a path in $G$ yields a sequence of exchanges in which, at each step, one point is removed from the current simplex and replaced by another point from $A\cup B$, while preserving the property that the origin lies in the convex hull. Based on computational experiments, Caoduro, Khodamoradi, Paat, and Shepherd~\cite{marco2024} proposed the following conjecture.

\begin{conj}[Caoduro, Khodamoradi, Paat, Shepherd]\label{conj:embracing}
The $0$-embracing exchange distance of two $0$-embracing $d$-dimensional simplices is at most $d+1$.
\end{conj}

Although the Hirsch conjecture is now known to be false in general, Conjecture~\ref{conj:embracing} would follow from its original formulation (see~\cite{dantzig}), since the polyhedron $P$ has $2d+2$ facets and dimension $2d+2-(d+1)=d+1$. Consequently, the Hirsch bound would guarantee the existence of a path of length at most $d+1$ between the corresponding vertices in the $1$-skeleton of $P$.

Conjecture~\ref{conj:embracing} remains open. The problem, however, resembles several long-standing conjectures concerning the symmetric exchange distance of basis-pairs in matroids, such as those proposed by Gabow~\cite{gabow1976decomposing}, White~\cite{white1980unique}, and Hamidoune~\cite{cordovil1993bases}; see Section~\ref{sec:gabow} for details. This observation naturally suggests that the problem may, in fact, be a special case of a more general, abstract question concerning matroids. 

\subsection{Basic notation}
\label{sec:not}

\paragraph{General notation.} For integers $k$ and $\ell$ with $k\le \ell$, we write $[k,\ell]={k,k+1,\dots,\ell}$; when $k=1$, we simply write $[\ell]=[1,\ell]$. Given a ground set $S$, a set $X\subseteq S$ and an element $y\in S$, the sets $X\setminus \{y\}$ and $X\cup \{y\}$ are often abbreviated as $X-y$ and $X+y$, respectively.  

\paragraph{Oriented matroids.} For basic definitions on oriented matroids, we refer to~\cite{bjorner1999oriented}.

An {\it oriented matroid} $M=(E,\mathcal{C})$ on a finite ground set $E$ can be defined via its collection $\mathcal{C}$ of {\it signed circuits}, where each $C\in\mathcal{C}$ is a pair $(C^+,C^-)$ of disjoint subsets of $E$ representing the elements with positive and negative signs, respectively, satisfying the following axioms.
\begin{enumerate}[label=(C\arabic*)]\itemsep0em
\item Nontriviality: $\emptyset\notin\mathcal{C}$.
\item Symmetry: If $C\in\mathcal{C}$, then $-C=(C^-,C^+)\in\mathcal{C}$.
\item Incomparability: If $C_1,C_2\in\mathcal{C}$ and $C_1\subseteq C_2$, then $C_1=C_2$ or $C_1=-C_2$.
\item Circuit elimination: If $C_1,C_2\in\mathcal{C}$, $C_1\neq -C_2$, and $e\in C_1^+\cap C_2^-$, then there exists $C_3\in\mathcal{C}$ such that $C_3^+\subseteq (C_1^+\cup C_2^+)-e$ and $C_3^-\subseteq (C_1^-\cup C_2^-)-e$.
\end{enumerate}
The circuit family of the {\it underlying unoriented matroid} $\underline{M}$ is obtained by ignoring the signs of the circuits. A {\it basis} of $M$ is a basis of $\underline{M}$, and the {\it fundamental circuit} of $e\notin B$ with respect to a basis $B$ is the unique circuit of $\underline{M}$ contained in $B+e$. Given a basis $B$ and an element $e\in E\setminus B$, let $C(B,e)$ denote the signed fundamental circuit of $e$ with respect to $B$, oriented so that $e$ has negative sign; we call this the {\it anchored fundamental circuit} of $e$ with respect to $B$. A basis $B$ is {\it $e$-embracing} if, in the anchored fundamental circuit of $e$ with respect to $B$, every element distinct from $e$ is positive. For two $e$-embracing bases $A$ and $B$, a sequence of bases $A=T_0,T_1,\dots,T_{q-1},T_q=B$ is an {\it $e$-embracing exchange sequence} if $|T_i\triangle T_{i-1}|=1$ and $T_i$ is $e$-embracing for each $i\in[q]$. One can then define the {\it $e$-embracing exchange distance} between two $e$-embracing bases as the minimum length of an $e$-embracing sequence that transforms $A$ into $B$ if such a sequence exists, and $+\infty$ otherwise.

\paragraph{Directed graphs.} For a directed graph $D=(V,E)$ with vertex set $V$ and arc set $E$, we write $n=|V|$ and $m=|E|$. Given a subset of arcs $F\subseteq E$, the sets of arcs in $F$ {\it entering} and {\it leaving} a vertex $v\in V$ are denoted by $\delta^{in}_F(v)$ and $\delta^{out}_F(v)$, respectively. For $u,v\in V$, a {\it directed path} from $u$ to $v$ is a sequence of arcs $(e_1,\dots,e_k)$, where $e_i=v_{i-1}v_i\in E$ for $i\in[k]$, with $v_0=u$, $v_k=v$, and all vertices $v_0,\dots,v_k$ being distinct. A {\it spanning tree} of $D$ is a subset $T\subseteq E$ whose underlying undirected graph forms a spanning tree on $V$, while a {\it cycle} of $D$ is a subset $C\subseteq E$ whose underlying undirected graph forms a cycle. For $u,v\in V$, we denote by $T[u,v]$ the \textit{unique path in $T$ from $u$ to $v$}, traversed starting at $u$ and ending at $v$, with arcs oriented as in $D$. Note that this path need not be a directed $u$-$v$ path. Any directed graph $D=(V,E)$ defines an oriented matroid by taking the graphic matroid on $E$ and encoding orientations through signed circuits. Each circuit corresponds to a cycle in the underlying undirected graph: choosing a traversal of the cycle, an edge is signed positive if its orientation in $D$ agrees with the traversal and negative otherwise. The resulting signed circuits define the {\it oriented graphic matroid} of $D$. Note that, given an arc $e=uv\in E$, a spanning tree $B$ is $e$-embracing precisely when $B[u,v]$ is a directed $u$-$v$ path in $B$. By a slight abuse of notation, we call a spanning tree $B$ \textit{$uv$-embracing} if $B[u,v]$ is a directed path, no matter whether $uv$ is an arc of $D$ or not.

\subsection{Matroidal extension}
\label{sec:our}

Any finite set $E\subseteq\bR^d$ defines an affine matroid $M=(E,\cC)$, where a subset $X\subseteq E$ is independent if its vectors are affinely independent. A minimal dependent set in $M$ is called a circuit. Each circuit can be assigned a sign pattern based on a nontrivial linear combination of its elements that sums to zero, resulting in a signed circuit. Reversing all signs yields another valid signed circuit, so each circuit gives rise to a pair of oppositely signed circuits. In this way, one obtains an oriented matroid from the underlying affine matroid. If $E$ is full-dimensional, then the bases of $M$ correspond to subsets of $E$ defining simplices; note that the rank of $M$ is $d+1$. A basis contains the origin precisely when, in the signed fundamental circuit of $\mathbf{0}$ oriented so that $\mathbf{0}$ has negative sign, every other element is positive. Equivalently, the vertices of the corresponding simplex then form a $\mathbf{0}$-embracing basis. Therefore, Conjecture~\ref{conj:embracing} asserts that the $\mathbf{0}$-embracing exchange distance between any two $\mathbf{0}$-embracing bases of the oriented affine matroid is at most the dimension plus one. As an extension to general oriented matroids, we propose the following.

\begin{conj}\label{conj:matroid1}
    The $e$-embracing exchange distance of two $e$-embracing bases of a rank-$r$ oriented matroid is at most $r$.   
\end{conj}

Closely related problems have been considered in the context of bounding the diameter of the cocircuit graph of an oriented matroid; see \cite{adler2020diameterscocircuitgraphsoriented}. That work studies linear and polynomial diameter bounds, and establishes positive results for oriented matroids of low rank and corank, as well as for small ground sets via computer verification. In contrast to Conjecture~\ref{conj:embracing}, a key difference in the abstract oriented matroid setting is that it remains open whether any exchange sequence between two bases meeting the conjecture's conditions exists at all, regardless of its length. This is in sharp contrast to the unoriented setting: if $A$ and $B$ are bases of an unoriented matroid $M$, the exchange axiom guarantees the existence of a sequence of exchanges transforming $A$ into $B$, and the shortest such sequence has length $|A \setminus B|$. No analogous statement holds in the oriented setting, not even for the oriented graphic matroid of directed graphs; see Example~\ref{ex:graphic}.

As a first step towards the general setting, we verify that the conjecture holds when $M$ is the oriented graphic matroid of a directed graph (Theorem~\ref{thm:directed}). In fact, we prove a slightly stronger statement: there exists an exchange sequence of length at most $r$ which is monotone in the sense that the intersection with the target basis does not decrease at every step. We also highlight the connection between our problem and certain conjectures on exchange distances of basis-pairs, and show that no analogous strengthening of the conjecture is possible (Example~\ref{ex:symmetric}).

\section{Directed graphs}
\label{sec:directed}

This section is devoted to proving Conjecture~\ref{conj:matroid1} for the oriented graphic matroid of a directed graph. This result is particularly interesting in light of the fact that, unlike in the unoriented setting, one cannot necessarily transform one basis $A$ into another $B$ by using $|A \setminus B|$ exchanges, as shown by the following example. 

\begin{ex}\label{ex:graphic}
    Let $D=(V,E)$ denote the directed graph where $V=\{s,t,u,v,w\}$ with arc set $E=\{su,sv,uw,vw,ut,vt\}$. Let $A=\{su,uw,vw,ut\}$ and $B=\{sv,uw,vw,vt\}$; see Figure~\ref{fig:nomon}. Note that transforming $A$ into $B$ with $|A\setminus B|$ exchanges requires, at each step, exchanging one arc from $A\setminus B$ with one from $B\setminus A$. Hence, the first arc to remove from $A$ must be either $su$ or $ut$. However, for either of these arcs, there is no arc in $B\setminus A$ whose addition would restore a directed $s$-$t$ path, showing that no exchange sequence of length two exists.

    It is worth mentioning that exchanging $vw$ with $sv$, then $ut$ with $vt$, and finally $su$ with $vw$ yields an $st$-embracing exchange sequence that transforms $A$ into $B$.
\end{ex}

\begin{figure}
\centering
\begin{minipage}[t]{0.48\textwidth}
  \centering
  \includegraphics[height=3cm]{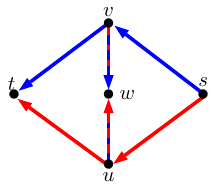}
  \captionof{figure}{Illustration of Example~\ref{ex:graphic}. Red and blue arcs correspond to $A$ and $B$, respectively. Note that $uw$ and $vw$ are contained in both spanning trees.}
  \label{fig:nomon}
\end{minipage}\hfill
\begin{minipage}[t]{0.48\textwidth}
  \centering
  \includegraphics[height=3cm]{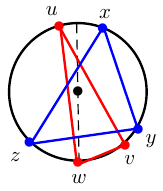}
  \captionof{figure}{Illustration of Example~\ref{ex:symmetric}. Red and blue triangles correspond to simplices $A$ and $B$, respectively. The only possible symmetric exchange uses $v$ and $y$.}
  \label{fig:symmetric}
\end{minipage}
\end{figure}

For ease of discussion, we restate the conjecture in graph-theoretic terms, obtaining a clean formulation about directed trees that avoids matroid terminology.

\begin{thm}\label{thm:directed}
Let $D=(V,E)$ be a directed graph and $s,t\in V$. Suppose $A$ and $B$ are $st$-embracing spanning trees of $D$. Then there exist $st$-embracing spanning trees $A=T_0,T_1,\dots,T_{q-1},T_q=B$ with $q\le n-1$ such that $|T_{i-1}\triangle T_i|=2$ and $|T_i\cap B|\ge |T_{i-1}\cap B|$ for all $i\in[q]$.
\end{thm}
\begin{proof}
We present an algorithm that constructs a sequence $A=T_0,T_1,\dots,T_{q-1},T_q=B$ satisfying the conditions of the theorem in two phases. The high-level idea is as follows: in the first phase, we modify $A$ by successively adding the edges of $B[s,t]\setminus A$, obtaining a spanning tree $A'$ with $A'[s,t]=G[s,t]$; in the second phase, we replace the edges of $A'\setminus B$ with those of $B\setminus A'$ in an arbitrary order, without altering this directed $s$-$t$ path.

Let $B[s,t]=(f_1,\dots,f_k)$, where $f_i=v_{i-1}v_i$ for $i\in[k]$, with the convention $v_0=s$ and $v_k=t$. For the first phase, we need the following technical claim.

\begin{cla}\label{cl:onestep}
Let $T$ be an $st$-embracing spanning tree of $D$ with $B[s,v_{i-1}]\subseteq T$ and $f_i\notin T$ for some $i\in[k]$. Then there exists $e\in T$ such that $T'=T-e+f_i$ is an $st$-embracing spanning tree with $B[s,v_i]\subseteq T'[s,v_i]$.
\end{cla}
\begin{proof}
    Clearly, $T-e+f_i$ is a spanning tree if and only if $e\in T[v_{i-1},v_i]$. Let $e$ be the last arc along this path not contained in $B[v_i,t]$. Since $f_i\notin T$, such an arc exists and has at least one endpoint on the path $B[v_i,t]$. If $e\notin T[s,t]$, then $T-e+f_i$ satisfies the conditions of the theorem. Otherwise, $e$ lies on the directed path $T[s,t]$. By the choice of $e$, this implies $e\in\delta^{in}(v_j)$ for some $j\in[i,k]$ with $B[i,j]\subseteq T$. Therefore, removing $T-e+f_i$ contains the directed paths $B[s,v_j]$ and $T[v_j,t]$, and hence a directed $s$-$t$ path.
\end{proof}

By repeated application of Claim~\ref{cl:onestep}, we obtain a sequence of at most $|B[s,t]|$ exchanges that transforms $A$ into a spanning tree $A'$ through $st$-embracing spanning trees, with $A'[s,t]=B[s,t]$. By the basis exchange axiom for the underlying unoriented graphic matroid of $D$, for every arc $e\in A'\setminus B$ there exists $f\in B\setminus A'$ such that $A'-e+f$ is again a spanning tree. Since such an exchange does not affect the directed path $A'[s,t]$, we can transform $A'$ into $B$ by performing exactly $|B\setminus A'|$ exchanges so that all intermediate spanning trees are $st$-embracing. 

Since at every step of the algorithm we add an arc from $B$, the intersection of each intermediate spanning tree with $B$ does not decrease in any step; however, it may stay the same when the deleted arc $e$ lies in $B$ as well. Therefore, concatenating the two exchange sequences defined above yields a sequence of at most $|B[s,t]|+|B\setminus A'|=|B|=n-1$ exchange steps through $st$-embracing spanning trees satisfying the conditions of the theorem.
\end{proof}

A nice feature of Theorem~\ref{thm:directed} is that it guarantees a polynomial-length exchange sequence and, when the two spanning trees are disjoint, allows transforming one basis into the other along a shortest possible sequence of exchanges that preserves the $st$-embracing property at every step. Indeed, for disjoint spanning trees, $n-1$ exchanges are necessary, and the theorem shows that this many steps also suffice.

For a finer distinction that applies more generally to oriented matroids, we call an embracing exchange sequence $A=T_0,T_1,\dots,T_{q-1},T_q=B$ \emph{monotone} if $|T_i\cap B|\ge |T_{i-1}\cap B|$ for all $i\in[q]$, and \emph{strictly monotone} if $q=|A\setminus B|$. Using this terminology, Theorem~\ref{thm:directed} establishes the existence of a monotone embracing exchange sequence in oriented graphic matroids. Note that if $A$ and $B$ are disjoint, then $r$ is a lower bound on the number of exchanges required to transform $A$ into $B$. Hence, for disjoint bases, Conjecture~\ref{conj:embracing} would imply the existence of a strictly monotone embracing exchange sequence. In light of Theorem~\ref{thm:directed}, it is natural to ask the rather optimistic strengthening of the conjecture: {\it Does a monotone $e$-embracing exchange sequence always exist?}

\section{Matroid exchange analogues}
\label{sec:gabow}

Let $M=(S,\cB)$ be a matroid where $\cB$ denotes the family of bases, and let $A,B\in\cB$. A {\it symmetric exchange} takes elements $e\in A\setminus B$ and $f\in B\setminus A$ such that both $A-e+f$ and $B-f+e$ remain bases. Gabow~\cite{gabow1976decomposing} conjectured that for a rank-$r$ matroid $M$ with two disjoint bases $A$ and $B$, it is always possible to transform the pair $(A,B)$ into $(B,A)$ by a sequence of exactly $r$ symmetric exchanges. This problem is closely related to a broader question posed by White~\cite{white1980unique}, who studied when two ordered pairs of bases $(A,B)$ and $(A',B')$ can be connected by a sequence of symmetric exchanges. The conditions $A\cap B=A'\cap B'$ and $A\cup B= A'\cup B'$, called {\it compatibility}, are clearly necessary for such a transformation, and White conjectured that they are also sufficient. Building on these ideas, Hamidoune~\cite{cordovil1993bases} proposed an optimization variant that unifies Gabow's and White's settings.

\begin{conj}[Hamidoune]\label{conj:hamidoune}
For any compatible basis pairs $(A,B)$ and $(A',B')$ in a rank-$r$ matroid, the first pair can be transformed into the second using at most $r$ symmetric exchanges. 
\end{conj}

There is a clear similarity between Conjectures~\ref{conj:matroid1} and~\ref{conj:hamidoune}, which raises the question of whether the former can be strengthened by allowing symmetric exchanges. More formally, given two $e$-embracing bases $A$ and $B$ of a rank-$r$ oriented matroid, one may ask whether there always exists a sequence of at most $r$ symmetric exchanges transforming $(A,B)$ into $(B,A)$ such that every intermediate pair in the sequence consists of $e$-embracing bases. We show that such a strengthening is not possible, even when the oriented matroid in question is an affine matroid and the bases are disjoint.

\begin{ex}\label{ex:symmetric}
Let $A=\{u,v,w\}$ and $B=\{x,y,z\}$ be the vertex sets of two $0$-embracing simplices in the plane whose vertices lie on the unit circle in the clockwise order $u,x,y,v,w,z$. Suppose further that the diameter passing through $w$ intersects the circle at a point lying between $u$ and $x$; see Figure~\ref{fig:symmetric}. These conditions imply that the only symmetric exchange preserving the $0$-embracing property of the simplices is the exchange of $v$ and $y$. Performing this exchange yields the simplices $A'=\{u,y,w\}$ and $B'=\{x,v,z\}$, for which the only possible symmetric exchange is again that of $y$ and $v$.
\end{ex}

\medskip
\paragraph{Acknowledgement.}

Kristóf Bérczi is grateful to Marco Caoduro for introducing him to the original problem on $0$-embracing sets, to András Imolay, Benjamin Schröter, and Lorenzo Vecchi for initial discussions on the matroidal extension, and to Frédéric Meunier for pointing out a connection with Hirsch’s conjecture. The research was supported by the Lend\"ulet Programme of the Hungarian Academy of Sciences -- grant number LP2021-1/2021, by the Ministry of Innovation and Technology of Hungary from the National Research, Development and Innovation Fund -- grant numbers ADVANCED 150556 and ELTE TKP 2021-NKTA-62, and by Dynasnet European Research Council Synergy project -- grant number ERC-2018-SYG 810115.

\bibliographystyle{abbrv}
\bibliography{simplex}

\end{document}